\documentclass[11pt, oneside, reqno]{amsart}

\newtheorem{thm}{Theorem}[section]
\newtheorem{lem}[thm]{Lemma}
\newtheorem{prop}[thm]{Proposition}

\theoremstyle{definition}
\newtheorem{defn}[thm]{Definition}

\newtheorem{obs}[thm]{Observation}
\newtheorem{notn}[thm]{Notation}

\theoremstyle{remark}

\numberwithin{equation}{section}

\usepackage{latexsym}
\usepackage{amsmath}
\usepackage{amsthm}
\usepackage{amssymb}
\usepackage{amsfonts}
\usepackage[all]{xy}
\usepackage{graphicx}
\usepackage{lscape}
\usepackage{verbatim}
\usepackage{wrapfig}
\usepackage{subfig}
\usepackage{setspace}
\usepackage{pinlabel}
\usepackage{enumerate, setspace, dsfont}
\usepackage{fancyhdr}
\usepackage{calc}
\usepackage{rotating}

\newcommand{\verteq}[0]{\begin{turn}{90} $=$\end{turn}}

\begin{document}

\title{Classifying Voronoi Graphs of Hex Spheres}
\author{Aldo-Hilario Cruz-Cota}
\address{Department of Mathematics, Grand Valley State University, Allendale, MI 49401-9401, USA}
\email{cruzal@gvsu.edu}


\date{\today}

\dedicatory{}

\begin{abstract}

\noindent A \emph{hex sphere} is a singular Euclidean sphere with four cones points whose cone angles are (integer) multiples of $\frac{2\pi}{3}$ but less than $2\pi$. Given a hex sphere $M$, we consider its Voronoi decomposition centered at the two cone points with greatest cone angles. In this paper we use elementary Euclidean geometry to describe the Voronoi regions of hex spheres and classify the Voronoi graphs of hex spheres (up to graph isomorphism).
\end{abstract}

\maketitle

\section{Introduction}
A surface is \emph{singular Euclidean} if it is locally modeled on either the Euclidean plane or a Euclidean cone. In this article we study a special type of singular Euclidean spheres that we call \emph{hex spheres}. These are defined as singular Euclidean spheres with four cone points which have cone angles that are multiples of $\frac{2\pi}{3}$ but less than $2\pi$. Singular Euclidean surfaces whose cone angles are multiples of $\frac{2\pi}{3}$ are mainly studied because they arise as limits at infinity of real projective structures. 

We now give examples of hex spheres. Consider a parallelogram $P$ on the Euclidean plane such that two of its interior angles equal $\pi/3$, while the other two equal $2\pi/3$. Such a parallelogram will be called a \emph{perfect parallelogram}. The  double $D$ of a perfect parallelogram $P$ is an example of a hex sphere. This example gives rise to a $3$-parameter family of hex spheres. To see this, let $\gamma$ be the simple closed geodesic in $D$ that is the double of a segment in $P$  that is perpendicular to one of the longest sides of $P$. Then two parameters of the family of hex spheres correspond to the lengths of two adjacent sides of $P$, and the other parameter corresponds to twisting $D$ along $\gamma$.

Let $M$ be a hex sphere. The Gauss-Bonnet Theorem implies that exactly two of the cone angles  of $M$ are equal to $\frac{4\pi}{3}$, while the other two are equal to $\frac{2\pi}{3}$. We consider the Voronoi decomposition of $M$ centered at the two cone points of angle $\frac{4\pi}{3}$. This decomposes $M$ into two cells, the Voronoi cells, which intersect along a graph in $M$, the Voronoi graph. 

We can now state our main results. 

\begin{thm} \label{thm-intro-classif-Vor-graph}
Let $M$ be a hex sphere, and let $\Gamma$ be the Voronoi graph of $M$ (with respect to the Voronoi decomposition of $M$ centered at the two cone points of angle $\frac{4\pi}{3}$). Then, up to graph isomorphism, $\Gamma$ is one of the graphs from Figure \ref{fig-classif-Vor-graph1}.

 \end{thm}

\begin{figure}[ht!]
\labellist
\small\hair 3pt

\endlabellist
\centering
  \includegraphics[width=0.45\textwidth]{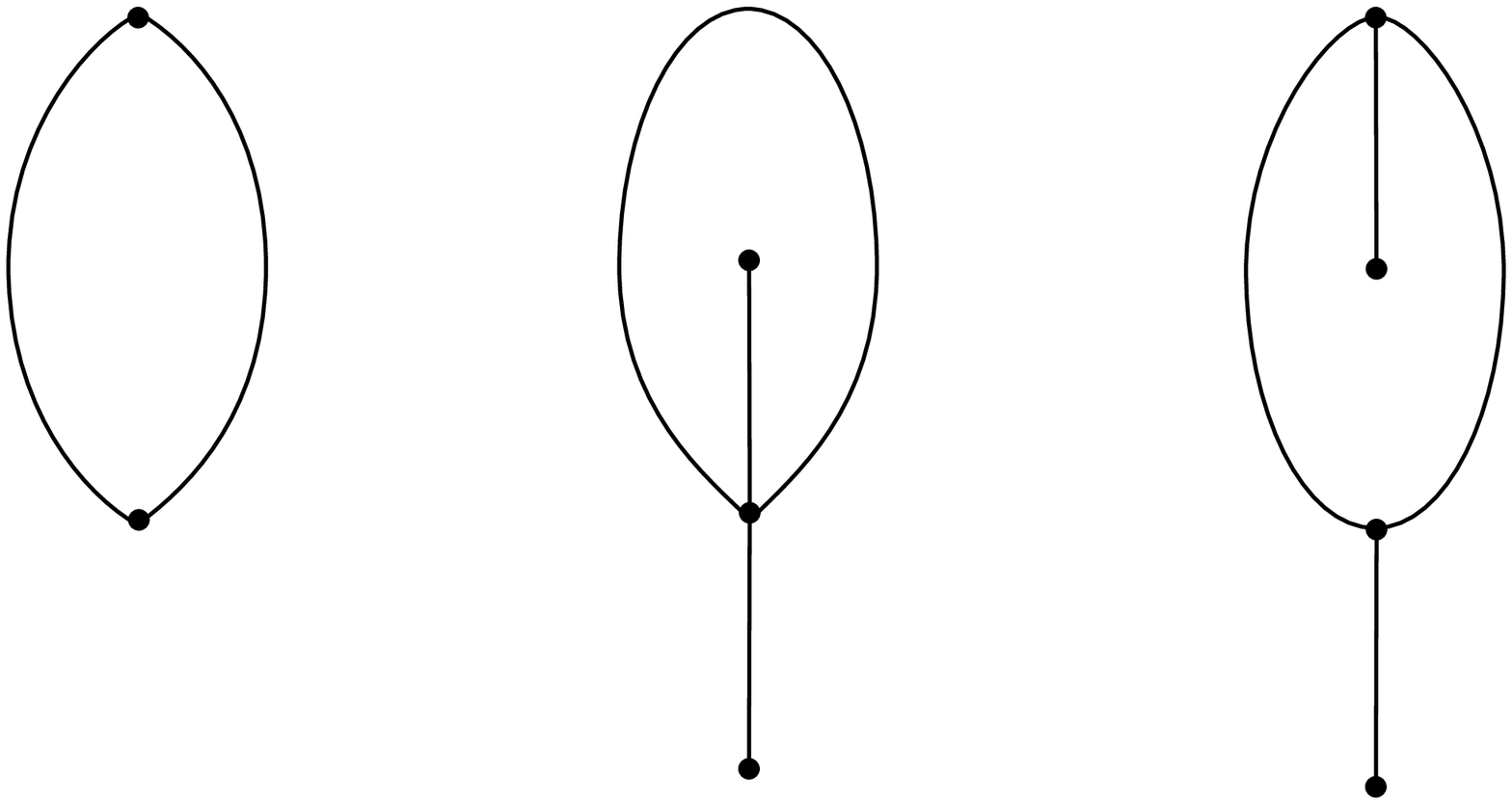}
  \caption{Classifying the Voronoi graphs of hex spheres}
  \label{fig-classif-Vor-graph1}
\end{figure}

\begin{thm} \label{thm-intro-describ-Vor-reg}
Let $M$ be a hex sphere and consider its Voronoi decomposition centered at the two cone points of angle $\frac{4\pi}{3}$. Then the two Voronoi regions of $M$ are isometric. These regions  embed isometrically in a Euclidean cone as convex, geodesic polygons, where the center of the Voronoi region corresponds to the vertex of the cone. Further, the hex sphere $M$ can be recovered from the disjoint union of the Voronoi regions by identifying pairs of edges on their boundaries according to one of $3$ possible combinatorial patterns (one pattern for each of the possible shapes of the Voronoi graphs, see Figure \ref{fig-classif-Vor-graph1}).
\end{thm}

We now sketch the proof of the main theorems. It is shown in \cite{Boileau} and \cite{Cooper} that the Voronoi region of a hex sphere centered at a cone point embeds isometrically in the tangent cone of the sphere at that point. The image of the Voronoi region in the cone is a convex, geodesic polygon that is star-shaped with respect to the vertex of the cone. The Gauss-Bonnet theorem gives numeric restrictions on the integers $p$ and $q$, which we define as the number of edges on the boundaries of the Voronoi regions. Then we do a case-by-case analysis of all possible values of $p$ and $q$, showing that only the most symmetric situation $p=q$ can occur. We also obtain that the numbers $p$ and $q$ can only be equal to either $2$, $3$ or $4$. These three possibilities give rise to the three possible Voronoi graphs from Figure \ref{fig-classif-Vor-graph1}, which gives a classification of Voronoi graphs of hex spheres. This proves Theorem \ref{thm-intro-classif-Vor-graph}. Then we analyze the cases $p=q=2$, $p=q=3$ and $p=q=4$ separately. In each of these cases, we use elementary Euclidean geometry to prove that the Voronoi regions of the hex sphere must be isometric. We also find the unique gluing pattern on the boundary of the Voronoi regions that allows to recover the singular hex sphere from its Voronoi regions. This concludes the sketch of the proof Theorem \ref{thm-intro-describ-Vor-reg}.

The author would like to thank his PhD adviser, Professor Daryl Cooper, for many helpful discussions. Portions of this work were completed at the University of California, Santa Barbara and Grand Valley State University.

\section{Singular Euclidean Surfaces} \label{sect-SES}

\begin{defn} (\cite{Troyanov1})  \label{defn_cone_manif}
A closed triangulated surface $M$ is \emph{singular Euclidean} if it satisfies the following properties:

\begin{enumerate}
\item For every 2-simplex $T$ of $M$ there is a simplicial homeomorphism $f_T$ of $T$ onto a non-degenerate triangle $f_T(T)$ in the Euclidean plane.
\item If $T_1$ and $T_2$ are two 2-simplices of $M$ with non-empty intersection, then there is an isometry $g_{12}$ of the Euclidean plane such that $f_{T_1}=g_{12}f_{T_2}$ on $T_1 \cap T_2$.
\end{enumerate}
\end{defn}

There is a natural way to measure the length of a curve $\gamma$ in a singular Euclidean surface $M$. This notion of length of curves coincides with the Euclidean length on each triangle of $M$ and it turns $M$ into a path metric space. That is, there is a \emph{distance function} $d_M$ on $M$ for which the distance between two points in $M$ is the infimum of the lengths of the paths in $M$ joining these two points.

\begin{defn} 
Let $M$ be a singular Euclidean surface $M$ and let $p$ be a point in $M$. The \emph{cone angle} of $M$ at $p$ is either $2\pi$ (if $p$ is not a vertex of $M$) or the sum of the angles of all triangles in $M$ that are incident to $p$ (if $p$ is a vertex of $M$). If $\theta$ is the cone angle of $M$ at $p$, then the number $k=2\pi-\theta$ is the (concentrated) \emph{curvature} of $M$ at $p$.
\end{defn}

The next definition generalizes the concept of tangent plane (see \cite{Burago}).

\begin{defn} (\cite{Cooper})
 Given a singular Euclidean surface $M$ and a point $p$ in $M$, the \emph{tangent cone} $T_pM$ of $M$ at $p$ is the union of the Euclidean tangent cones to all the 2-simplices containing $p$.  The cone $T_pM$ is isometric to a Euclidean cone of angle equal to the cone angle of $M$ at $p$. The vertex of the cone $T_pM$ will be denoted by $v_p$.
\end{defn}

A point $p$ in a singular Euclidean surface $M$ is called \emph{regular} if its cone angle equals $2\pi$. Otherwise it is called a \emph{singular} point or a \emph{cone point}. 
The \emph{singular locus} $\Sigma$ is the set of all singular points in $M$.

\section{Two Theorems from Differential Geometry}

A \emph{geodesic} in a singular Euclidean surface $M$ is a path in $M$ that is \emph{locally}
length-minimizing. A \emph{shortest geodesic} $\gamma$ is a path in
$M$ that is \emph{globally} length minimizing (i.e, the distance
between the endpoints of $\gamma$ is equal to the length of
$\gamma$). The geodesics in this article will always be parametrized by arc-length.

The following two statements are the classical theorems of Hopf-Rinow and Gauss-Bonnet adapted to our context.

\begin{thm} (\cite{Hodgson}) \label{H-R}
Let $M$ be a complete, connected singular Euclidean surface. Then every pair of points
in $M$ can be joined by a shortest geodesic in $M$. 
\end{thm}

\begin{thm} (\cite{Cooper}) \label{G-B}
Let $M$ be a singular Euclidean surface, and let $F$ be a compact
region of $M$. Assume that the interior of $F$ contains $n$ cone
points with cone angles $\theta_{1}$, $\theta_{2}$, ...,
$\theta_{n}$, and that the boundary of $F$ is a geodesic polygon
with corner angles $\alpha_{1}$, $\alpha_{2}$,...$\alpha_{k}$. Then
\[\sum^{n}_{i=1}(2\pi-\theta_{i})+\sum^{k}_{j=1}(\pi-\alpha_{j})=2\pi\chi(F),\]
where $\chi(F)$ is the Euler characteristic of $F$.
\end{thm}

\section{Hex Spheres}

\begin{defn} \label{hex_spher}
A \emph{hex sphere} is an oriented singular Euclidean sphere with $4$ cone points whose cone angles are integer multiples of $\frac{2\pi}{3}$ but less than $2\pi$.
\end{defn}

Examples of hex spheres are given in the introduction of this paper.

\noindent $\diamond$ \emph{Why cone angles that are multiples of $\frac{2\pi}{3}$?}
Singular Euclidean surfaces with these cone angles arise naturally as limits at infinity of real projective structures. Real projective structures have been studied extensively by many authors, including \cite{Goldman}, \cite{Choi-Goldman}, \cite{Loftin}, \cite{Labourie} and \cite{Hitchin}.

\noindent $\diamond$ \emph{Why $4$ cone points?} The following lemma shows that there is only one singular Euclidean sphere with $3$ cone points whose cone angles satisfy the numeric restrictions we are interested in. This suggests studying the next simplest case (when the singular sphere has $4$ cone points).

\begin{lem} \label{lem-hex-3-cone}
Let $M$ be a singular Euclidean sphere with $k$ singular points and assume that the cone angle of $M$ at every singular point is an integer multiple of $\frac{2\pi}{3}$. Then $k \geq 3$, and if $k=3$ then $M$ is the double of an Euclidean equilateral triangle.
\end{lem}

The proof of Lemma \ref{lem-hex-3-cone} follows from Theorem \ref{G-B} and Proposition 4.4 from \cite{Cooper}. Theorem \ref{G-B} also gives the following:

\begin{lem} \label{sizes_con-angles}
Exactly two of the cone angles of a hex sphere equal $\frac{2\pi}{3}$ while the other two equal $\frac{4\pi}{3}$.
\end{lem}

From now on, we will use the following notation:

\begin{itemize}
\item $M$ will be a hex sphere.
\item $a$ and $b$ will denote the two cone points in $M$ of angle $\frac{4\pi}{3}$.
\item  $Equidist(a,b)$  will be the set of all points in $M$ which are \emph{equidistant} from $a$ and $b$.
\item $c$ and $d$ will denote the two cone points in $M$ of angle $\frac{2\pi}{3}$.
\item $\Sigma=\{a,b,c,d\}$ will denote the singular locus of $M$.
\end{itemize}

\begin{thm} \label{holonomy_argument} (The holonomy argument)
With the previous notation, $d_M(a,d)=d_M(b,c)$  and $d_M(a,c)=d_M(b,d)$.
\end{thm}

\begin{proof}
We prove only that $d_M(a,d)=d_M(b,c)$. Choose a base triangle $T_0$ in the triangulation of $M$, a base point $x_0 \in T_0 \, \backslash \, \Sigma$ and an isometry $f_0$ from $T_0$ to
the Euclidean plane $\mathbb{E}^2$. Consider the developing map
$dev \colon \widehat{M} \to \mathbb{E}^2$ associated to the pair
$(T_0,f_0)$, and let $hol \colon G \cong \pi_1(M \, \backslash \, \Sigma, x_0)
\to Isom(\mathbb{E}^2)$ be the corresponding holonomy homomorphism (see \cite{Troyanov1}). For each singular point $s$, let
$\alpha_s$ be a loop in $ M \, \backslash \, \Sigma$ based at $x_0$
which links only the cone point $s$, so that the homotopy classes of the loops $\alpha_a$, $\alpha_b$, $\alpha_c$ and
$\alpha_d$ generate the group $G \cong \pi_1(M \, \backslash \, \Sigma, x_0)$.

Since $M$ is a sphere, then $hol([\alpha_a \cdot \alpha_c])=hol([\alpha_b \cdot \alpha_d])$, where $[\sigma]$ and $\cdot$  denote the homotopy class of the path $\sigma$ and concatenation of paths (respectively). Also, $hol([\alpha_a])$, $hol([\alpha_b])$, $hol([\alpha_c])$ and $hol([\alpha_d])$ are rotations on the Euclidean plane $\mathbb{E}^2$, the first two of angle $\frac{4\pi}{3}$ and the last two of angle $\frac{2\pi}{3}$. For each singular point $p$, let $F_p$ be the fixed point of the rotation $hol([\alpha_p])$. Using Euclidean geometry, the reader can check that the isometry $hol([\alpha_a \cdot \alpha_c])$ is a translation on $\mathbb{E}^2$ with translational length $r=\sqrt{3} \, d_{\mathbb{E}^2}(F_a, F_c)$, where $d_{\mathbb{E}^2}$ denotes the distance on $\mathbb{E}^2$. Since $d_M(a,c)=d_{\mathbb{E}^2}(F_a, F_c)$ then the translational length of $hol([\alpha_a \cdot \alpha_c])$ equals $\sqrt{3} \, d_M(a,c)$. Similarly, the translational length of $hol([\alpha_b \cdot \alpha_d])$ equals $\sqrt{3} \, d_M(b,d)$. Thus, $hol([\alpha_a \cdot \alpha_c])=hol([\alpha_b \cdot \alpha_d])$ implies that $d_M(a,c)=d_M(b,d)$.
\end{proof}

\section{The Voronoi Regions of $M$ and the Voronoi Graph $\Gamma$}

\begin{defn} \label{defn-vor-reg}
 The (open) \emph{Voronoi region} $Vor(a)$ centered at $a$ is the set of points in $M$ consisting of:
\begin{itemize}
 \item the cone point $a$, and
 \item all non-singular points $x$ in $M$ such that 
  \begin{enumerate}
   \item $d_M(a,x) < d_M(b,x)$ and
   \item there exists a \emph{unique} shortest geodesic from $x$ to $a$.
  \end{enumerate}
\end{itemize}
\end{defn}

The (open) Voronoi region $Vor(b)$ centered at $b$ is defined by swapping the roles of $a$ and $b$ in Definition \ref{defn-vor-reg}. 

\begin{prop} \label{cut-loc} The Voronoi regions $Vor(a)$ and $Vor(b)$ are locally polyhedral and all of their corner angles are less than or equal to $\pi$.
\end{prop}

We omit the proofs of Proposition \ref{cut-loc} and Lemma \ref{anal_cut_loc} below, as they use the same ideas from the proof of Proposition 3.14 in \cite{Cooper}.

The complement of the Voronoi regions $Vor(a)$ and $Vor(b)$ in $M$ is called the \emph{cut locus} $Cut(M)$ of $M$. 

\begin{lem} \label{anal_cut_loc} The singular locus $Cut(M)$ is a graph embedded in $M$ such that:
\begin{itemize}
 \item its edges are geodesics in $M$;
 \item its vertex set contains $\Sigma \cap Cut(M)$;
 \item the degree of a vertex $v$ of $Cut(M)$ is equal to the (finite) number of shortest geodesics in $M$ from $v$ to the set $\{a,b\}$.
\end{itemize}
\end{lem}

Each Voronoi region of $M$ embeds in a certain tangent cone to $M$.

\begin{prop} (\cite{Boileau}, \cite{Cooper}) \label{Vor_emb_in_con} Let $p$ be either $a$ or $b$ and consider the map $f_p \colon Vor(p) \to T_{p}M$ defined by $f_p(x)=[\gamma_{x}'(0), d_{M}(p,x)]$ for $x \in Vor(p)$, where $\gamma_{x}$ is the unique shortest geodesic in $M$ from $p$ to $x$. Then the map $f_p$ is an isometric embedding and its image is the interior of a convex geodesic polygon in $T_{p}M$ that is star-shaped with respect to the vertex $v_p$ of $T_pM$.
\end{prop}

We will use the following notation:
\begin{itemize}
 \item $A$ will be the closure of $f_a(Vor(a))$ in $T_aM$.
 \item $B$ will be the closure of $f_b(Vor(b))$ in $T_bM$.
 \item $\sqcup$ denotes the disjoint union of sets.
 \item $\partial(\cdot)$ and  $\text{int}(\cdot)$ denote the boundary and the interior of $\cdot$ in the appropriate tangent cone.
\end{itemize}

By Lemma \ref{anal_cut_loc} and Proposition \ref{Vor_emb_in_con}, the hex sphere $M$ can be recovered from $A$ and $B$ by identifying the edges of $\partial A$ and $\partial B$ in pairs (an edge of $\partial A$ can be identified to another edge in $\partial A$). Therefore, there is a surjective quotient map $\pi \colon A \sqcup  B \to M$, which is injective in $\text{int}(A) \sqcup \text{int}(B)$. By abuse of notation, the disks $A$ and $B$ will also be called Voronoi cells.

\begin{defn} \label{defn_gamma}
The \emph{Voronoi graph} $\Gamma$ of a hex sphere $M$ is defined by \[\Gamma=\pi( \partial A \sqcup \partial B).\]
\end{defn}

It is easy to see that the graph $\Gamma$ is connected, that it contains the set $Equidist(a,b)$, and that the cone points $c$ and  $d$ are vertices of $\Gamma$.

The proof of the next proposition follows from the definitions.

\begin{prop} \label{equidist(a,b)} If $x \in \pi(\partial A)$, then $d_M(a,x) \leq d_M(b,x)$. Further, the set $Equidist(a,b)$ contains $\pi (\partial A) \cap \pi (\partial B)$.
\end{prop}

Lemma \ref{anal_cut_loc} and Theorem \ref{H-R} immediately imply the following lemma.

\begin{lem} \label{lem_vert_deg_1}
 Let $v$ be a vertex of the graph $\Gamma$. If the degree of $v$ in $\Gamma$ is equal to $1$, then $v$ is \emph{not} a point in $Equidist(a,b)$.
\end{lem}

Let $v \in M$ be a vertex of $\Gamma$. By the proof of  Proposition 3.14 in \cite{Cooper}, there is a neighborhood of $v$ in $M$ that is obtained by gluing some corners $C_1, C_2, \cdots , C_k$ of $A \sqcup B$ along edges, where $k$ is the degree
of the vertex $v$ of the graph $\Gamma$. Let $\theta_1, \theta_2, \cdots
, \theta_k$ be the angles at the corners $C_1, C_2, \cdots , C_k$
(respectively). Since the polygons $A$ and $B$ are convex, then
$\theta_i<\pi$ for each $i$, and thus we obtain that the \emph{cone
angle} of $v$ $=\sum^k_{i=1} \theta_i < k \pi$. In particular, if
$v$ is a non-singular point in $M$, then the cone angle at $v$ is
equal to $2 \pi$ , and so we obtain that $k \geq 3$. This shows the
following:

\begin{obs} \label{degree_leq_2}
If $v \in M$ is a vertex of the graph $\Gamma$ of degree $\leq 2$,
then $v$ is a cone point of angle $\frac{2 \pi}{3}$. In particular, the graph
$\Gamma$ contains at most two vertices of degree $\leq 2$.
\end{obs}

For the rest of this article we will use the following notation:
\begin{itemize}
\item $n \geq 2$ will be the number of vertices of the graph $\Gamma$;
\item $p$, $q$ will be the number of edges on $\partial A$, $\partial B$ (respectively).
\end{itemize}

Using Theorem \ref{G-B} and elementary combinatorics we obtain the following proposition.

\begin{prop} \label{num_cond}
The numbers $p$, $q$ and $n$ satisfy the following:
\begin{itemize}
\item $p \geq 2$, $q \geq 2$ and $p+q$ is even;
\item $n=\frac{p+q}{2}$ equals the number of edges of the graph
$\Gamma$;
\item $p+q \leq 8$.
\end{itemize}
\end{prop}

We now prove that there is \emph{only} one cycle in the graph $\Gamma$.

\begin{thm} \label{gamma_prop}
The graph $\Gamma$  contains a unique cycle.
\end{thm}

\begin{proof}
By Proposition \ref{num_cond}, the number of vertices of the
connected graph $\Gamma$ equals the number of its edges. Therefore, $\Gamma$ is \emph{not} a tree and so it must contain a cycle. To prove that this cycle is unique, we apply Alexander's Duality to the graph $\Gamma$, which is embedded in the sphere (see \cite{Hatcher} for a statement of Alexander's Duality). If $H$ stands for
the \emph{reduced} homology or cohomology with integer coefficients, then we obtain that $H_1(\Gamma) \cong H^0(S^2 \,
\backslash \, \Gamma) \cong Hom(H_0(S^2 \, \backslash \, \Gamma),
\mathbb{Z})$. Since  $S^2 \, \backslash \, \Gamma$ has two connected
components (the open Voronoi cells), then
$H_0(S^2 \, \backslash \, \Gamma) \cong \mathbb{Z}$, which
implies that $H_1(\Gamma) \cong \mathbb{Z}$. This means that
$\Gamma$ has a unique cycle.
\end{proof}

\begin{notn} \label {plq}
By relabeling the polygons $A$ and $B$ if necessary, we may (and will) suppose that $p \leq q$.
\end{notn}

\section{Analyzing the Possible Values of $p$ and $q$} \label{chapt_Possib_Values_pq}

Proposition \ref{num_cond} and Notation \ref{plq} imply the following:
\begin{obs} \label{obs-pos-val-p}
 The \emph{only} possible values for $p$ are $2$, $3$ and $4$
\end{obs}
If $p$ is either $a$ or $b$, then, by abuse of notation, the vertex $v_p$ of the cone $T_pM$ will also be denoted by $p$. 

We now do a case-by-case analysis of all possible values for $p$ and $q$.

\subsection{The Case $p=2$} \label{subsect-case-p2}

By Proposition \ref{num_cond} and Notation \ref{plq}, the only possible values for $q$ in this case are $2$, $4$ and $6$. We will show that the case $q=2$ is the \emph{only} one that can occur.

\begin{lem} \label{lem_p2_q2}
 Suppose that $p=2$ and $q=2$. Then the graph $\Gamma$ is a cycle on $2$ vertices. Moreover, the disks $A$ and $B$ are  isometric, and each of them satisfies the following:
\begin{enumerate}
 \item its interior contains the vertex of the cone $T_aM$;
 \item its boundary consists of two geodesics in $T_aM$;
 \item each  of its two corner angles equals $\pi/3$.
\end{enumerate}
Further, the hex sphere $M$ is the double of $A$ and it can be recovered from the (planar) isometric polygons $P_A$ and $P_B$ from Figure \ref{summ_2}  by identifying pairs of edges on their boundaries as shown in Figure \ref{summ_2}. 
\end{lem}

\begin{figure}[ht!]
\labellist
\small\hair 2pt

\pinlabel $\frac{4\pi}{3}$ [t] at 86 591 
\pinlabel $\frac{\pi}{3}$ [b] at 95 432
\pinlabel $P_A$ at 96 508

\pinlabel $\frac{4\pi}{3}$ [t] at 495 591 
\pinlabel $\frac{\pi}{3}$ [b] at 504 432
\pinlabel $P_B$ at 505 508

\pinlabel $\Gamma$ at 290 287
\endlabellist
\centering
  \includegraphics[width=0.55\textwidth]{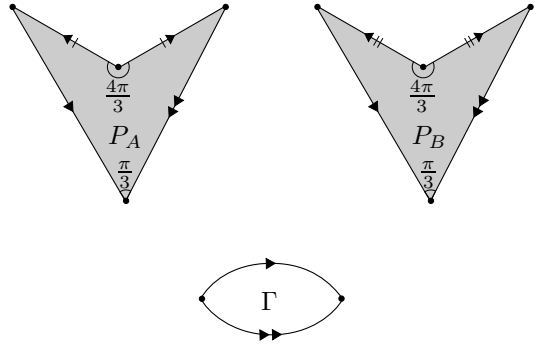}
  \caption{Recovering $M$ from $P_A$ and $P_B$ when $p=q=2$}
  \label{summ_2}
\end{figure}

\begin{proof}
Since $p=2$ and $q=2$, then both $A$ and $B$ are \emph{bigons}. Let $d'$ and $c'$ be the two vertices of $A$ with corner angles $\theta$ and $\phi=2\pi/3-\theta$, respectively (see Figure \ref{double}). Let $c''$ and $d''$ be the vertices of $B$ with $\pi(c')=\pi(c'')=c$ and $\pi(d')=\pi(d'')=d$. Then the corner angles of $B$ at $d''$ and $c''$ equal $2\pi/3-\theta=\phi$ and  $2\pi/3-\phi=\theta$, respectively (see Figure \ref{double}).

Consider the quotient map $\pi \colon A \sqcup  B \to M$ that identifies the edges $\partial A \sqcup \partial B$ in pairs to obtain $M$. Since the graph $\Gamma$ is connected, then
there is an edge on $\partial A$ that gets identified to an edge on $\partial B$.
Let $x$ be the common length of these two edges. The remaining two edges
on $\partial A \sqcup \partial B$ get identified
between themselves. Let $y$ be the common length of these two edges.  Since the map $\pi$ is $1$-$1$ on $\text{int}(A) \sqcup \text{int}(B)$ and $\pi(\partial A)=\pi(\partial B)$, then standard topological arguments show that the restriction of the map $\pi$ to either $A$ or $B$ is a topological embedding. The hex sphere $M$ can be recovered from $A$ and $B$ by
identifying their boundaries according to the gluing pattern from
Figure \ref{double}.

\begin{figure}[ht!]
\labellist
\small\hair 2pt

\pinlabel $a$ [b] at 70 605
\pinlabel $d'$ [b] at 74 749
\pinlabel $\theta$ [t] at 70 730 
\pinlabel $c'$ [t] at 70 456
\pinlabel $A$ at 118 604
\pinlabel $\phi$ [b] at 70 477

\pinlabel $b$ [b] at 498 605
\pinlabel $d''$ [b] at 504 749
\pinlabel $\phi$ [t] at 498 732
\pinlabel $c''$ [t] at 504 456
\pinlabel $B$ at 546 604
\pinlabel $\theta$ [b] at 499 476

\pinlabel $x$ [r] at -23 604
\pinlabel $y$ [l] at 163 604

\pinlabel $x$ [r] at 405 604
\pinlabel $y$ [l] at 591 604
\endlabellist
  \centering
  \includegraphics[width=0.4\textwidth]{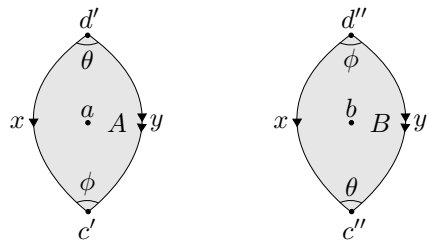}
   \caption{Recovering $M$ from $A$ and $B$}
    \label{double}
 \end{figure}

The graph $\Gamma$ has $2$ vertices and $2$ edges by Proposition \ref{num_cond}.
Since $\pi$ restricted to $A$ is an embedding, then $\pi(\partial A)$ is a \emph{cycle} graph on $2$ vertices. Thus $\pi(\partial A)$ is a subgraph of $\Gamma$ that has $2$ vertices and $2$ edges and therefore it coincides with $\Gamma$.

Since $\pi(\partial A)=\pi(\partial B)$, then the graph $\pi(\partial A)=\pi(\partial A) \cap \pi(\partial B)$, and hence Proposition \ref{equidist(a,b)} implies that $\pi(\partial A) \subset Equidist(a,b)$. In particular,
$c, d \in Equidist(a,b)$. Combining this with Theorem \ref{holonomy_argument},
we obtain that $d_M(a,d)=d_M(b,c)=d_M(a,c)=d_M(b,d)$, which implies that $A$ satisfy $(1)-(3)$ from the statement of the lemma.

Since the vertex $a$ of the cone $T_aM$ lies in the interior of $A \subset T_aM$, then there is a unique shortest geodesic in $A$ from $a$ to $d'$. Cutting $A$ along this geodesic we get a planar polygon $P_A$. Similarly, cutting $B$ along the unique shortest geodesic in $B$ from $b$ to $d''$, we get a planar polygon $P_B$. The polygons $P_A$ and $P_B$ are isometric because $d_M(a,d)=d_M(b,c)=d_M(a,c)=d_M(b,d)$, $d_{P_A}(c',d'_1)=d_{P_B}(c'',d''_1)$ and $d_{P_A}(c',d'_2)=d_{P_B}(c'',d''_2)$. Therefore, the disks $A$ and $B$ are also isometric.

The last assertion of the statement of the lemma follows from Figure \ref{double}, which shows how to recover $M$ from $A$ and $B$ by identifying pair of edges on their boundaries.
\end{proof}

We now show that the subcase $p=2$ and $q=4$ is \emph{impossible}.

\begin{lem} \label{lem_p2_q4}
The case $p=2$ and $q=4$ can not occur.
\end{lem}

\begin{proof}
By Theorem \ref{gamma_prop}, $\Gamma$ is a graph embedded on the sphere that contains exactly one cycle. Also, by Proposition \ref {num_cond}, the number of vertices of the graph $\Gamma$ equals the number of edges of $\Gamma$, which equals $\frac{p+q}{2}=\frac{2+4}{2}=3$.

Since the graph $\Gamma$ is connected, then there is an edge $a_1$ of $\partial A$ that gets identified to an edge $b_1$ of $\partial B$. Let $a_2$ be the other edge of $\partial A$, which gets identified to an edge $b_2$ of $\partial B$ (other than $b_1$). Let $b_3$ and $b_4$ be the other two edges of $\partial B$, which get identified between themselves.

Consider the space $X$ obtained from $A \sqcup B$ by identifying $a_1$ with $b_1$. We have 3 cases, depending on the location of the edge $b_2$ on $\partial B$. These cases correspond to the $3$ diagrams on the right of Figure \ref{p2q4_3}. We will show that in each of these cases we get a contradiction, and thus the case $p=2$, $q=4$ is impossible.

\begin{figure}[ht!]
\labellist
\small\hair 2pt

\pinlabel $\frac{4\pi}{3}$ [t] at 116 427
\pinlabel $\frac{4\pi}{3}$ [t] at 199 427
\pinlabel $a_1$ [r] at 164 449
\pinlabel $\verteq$ [r] at 160 437
\pinlabel $b_1$ [r] at 164 423
\pinlabel $a_2$ [l] at 237 436
\pinlabel $A$ at 204 470
\pinlabel $B$ at 110 470
\pinlabel $X$ [b] at 151 510
\pinlabel {\LARGE $\circlearrowleft$} at 199 436
\pinlabel {\LARGE $\circlearrowleft$} at 116 436

\pinlabel $\frac{4\pi}{3}$ [t] at 398 427
\pinlabel $\frac{4\pi}{3}$ [t] at 481 427
\pinlabel $a_1$ [r] at 446 449
\pinlabel $\verteq$ [r] at 442 437
\pinlabel $b_1$ [r] at 446 423
\pinlabel $a_2$ [l] at 519 436
\pinlabel $A$ at 486 470
\pinlabel $B$ at 392 470
\pinlabel $b_2$ [r] at 347 436
\pinlabel {\LARGE $\circlearrowleft$} at 481 436
\pinlabel {\LARGE $\circlearrowleft$} at 398 436

\pinlabel $\frac{4\pi}{3}$ [t] at 398 607
\pinlabel $\frac{4\pi}{3}$ [t] at 481 607
\pinlabel $a_1$ [r] at 446 629
\pinlabel $\verteq$ [r] at 442 617
\pinlabel $b_1$ [r] at 446 603
\pinlabel $a_2$ [l] at 519 616
\pinlabel $A$ at 486 650
\pinlabel $B$ at 392 650
\pinlabel $b_2$ [b] at 414 676
\pinlabel {\LARGE $\circlearrowleft$} at 481 616
\pinlabel {\LARGE $\circlearrowleft$} at 398 616

\pinlabel $\frac{4\pi}{3}$ [t] at 398 247
\pinlabel $\frac{4\pi}{3}$ [t] at 481 247
\pinlabel $a_1$ [r] at 446 269
\pinlabel $\verteq$ [r] at 442 257
\pinlabel $b_1$ [r] at 446 243
\pinlabel $a_2$ [l] at 519 256
\pinlabel $A$ at 486 290
\pinlabel $B$ at 392 290
\pinlabel $b_2$ [t] at 414 197
\pinlabel {\LARGE $\circlearrowleft$} at 481 256
\pinlabel {\LARGE $\circlearrowleft$} at 398 256

\endlabellist
\centering
\includegraphics[width=0.7\textwidth]{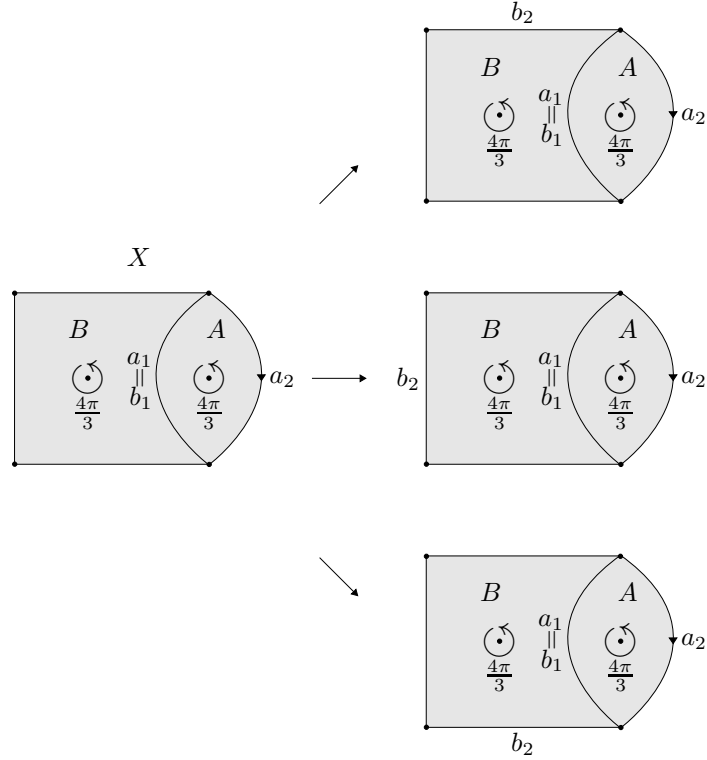}
\caption{The possibilities for the edge $b_2$}
\label{p2q4_3}
\end{figure}

$\bullet$ \textbf{Case I:} \emph{The edge $b_2$ is located as in the top right diagram of Figure \ref{p2q4_3}}. We can assume that $b_3$ and $b_4$ have the orientations indicated in Figure \ref{p2q4_5}. Therefore, $\Gamma$ has $3$ vertices, which have degrees $1$, $2$ and $3$. By Observation \ref{degree_leq_2}, exactly one of $c$ and $d$, say $c$, is the vertex of degree $1$ and the other, $d$, is the vertex of degree $2$ (see Figure \ref{p2q4_6}).

\begin{figure}[ht!]
\labellist
\small\hair 2pt

\pinlabel $\frac{4\pi}{3}$ [t] at 116 427
\pinlabel $\frac{4\pi}{3}$ [t] at 199 427
\pinlabel $a_1$ [r] at 164 449
\pinlabel $\verteq$ [r] at 160 437
\pinlabel $b_1$ [r] at 164 423
\pinlabel $b_2$ [b] at 134 498
\pinlabel $a_2$ [l] at 237 436
\pinlabel $b_3$ [r] at 63 436
\pinlabel $b_4$ [t] at 134 373
\pinlabel {\LARGE $\circlearrowleft$} at 199 436
\pinlabel {\LARGE $\circlearrowleft$} at 116 436
\vspace{12pt}
\endlabellist
 \centering
 \subfloat[]{\label{p2q4_5}\includegraphics[width=0.28\textwidth]{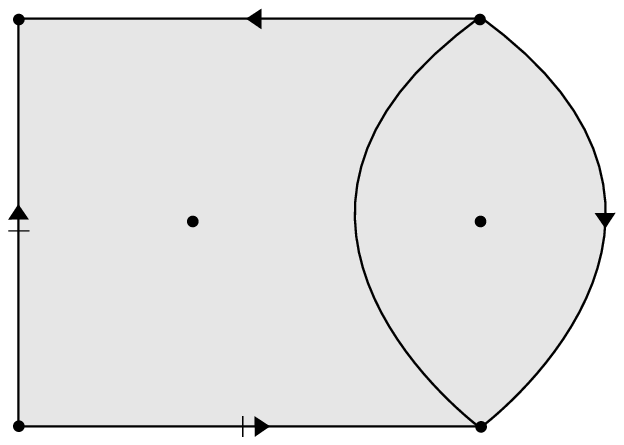}}
\hspace{0.25\textwidth}
\labellist
\small\hair 2pt
\pinlabel $d$ [b] at 342 683
\pinlabel $c$ [t] at 342 220
\vspace{12pt}
\endlabellist
  \subfloat[]{\label{p2q4_6}\includegraphics[width=0.1\textwidth]{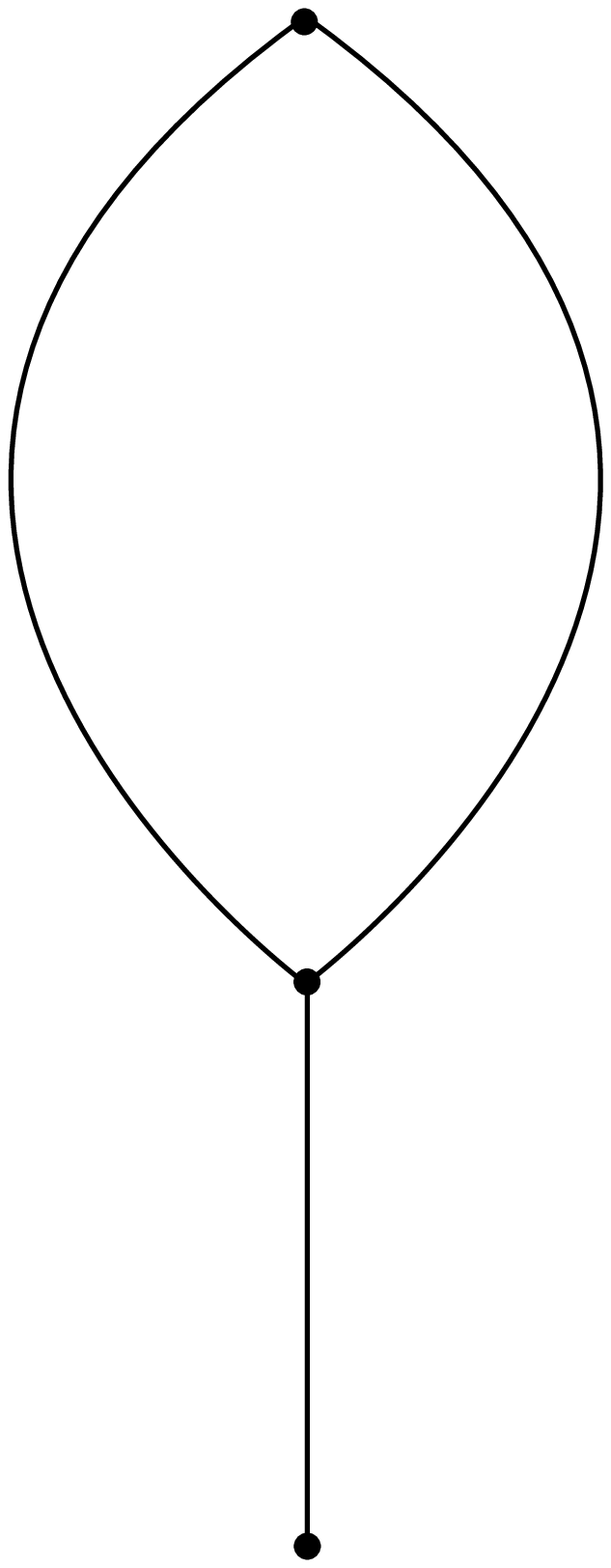}}
 \caption{Gluing pattern on $\partial X$ and the graph $\Gamma$ in Case I}
\end{figure}

Since  $d \in \pi(\partial A) \cap \pi(\partial B) \subset Equidist(a,b)$, then Theorem \ref{holonomy_argument} implies that $c \in Equidist(a,b)$. This contradicts Lemma \ref{lem_vert_deg_1}.

$\bullet$ \textbf{Case II:} \emph{The edge $b_2$ is located as in the middle right diagram of Figure \ref{p2q4_3}}. Since $M$ is an orientable surface $M$, then we can assume that the orientations of the edges $a_2$, $b_2$, $b_3$ and $b_4$ are those from Figure \ref{p2q4_7}. In particular, $M$ is homeomorphic to a torus, which contradicts that $M$ is a hex sphere.

\begin{figure}[ht!]
\labellist
\small\hair 2pt

\pinlabel $\frac{4\pi}{3}$ [t] at 116 427
\pinlabel $\frac{4\pi}{3}$ [t] at 199 427
\pinlabel $a_1$ [r] at 164 449
\pinlabel $\verteq$ [r] at 160 437
\pinlabel $b_1$ [r] at 164 423
\pinlabel $b_3$ [b] at 134 498
\pinlabel $a_2$ [l] at 237 436
\pinlabel $b_2$ [r] at 63 436
\pinlabel $b_4$ [t] at 134 374
\pinlabel {\LARGE $\circlearrowleft$} at 199 436
\pinlabel {\LARGE $\circlearrowleft$} at 116 436
\vspace{12pt}
\endlabellist
  \centering
\subfloat[Case II]{\label{p2q4_7}\includegraphics[width=0.28\textwidth]{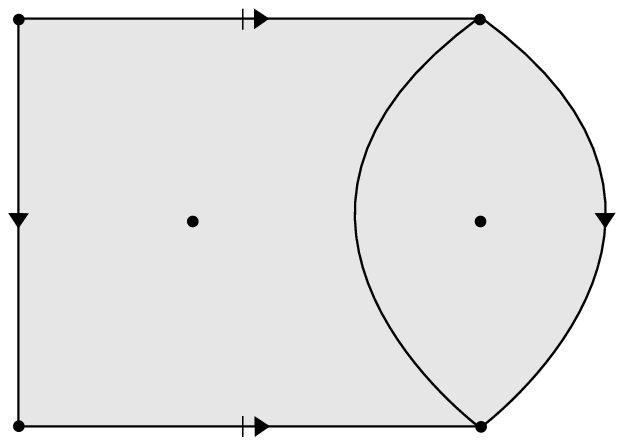}}
\hspace{0.2\textwidth}
\labellist
\small\hair 2pt
\pinlabel $\frac{4\pi}{3}$ [t] at 116 427
\pinlabel $\frac{4\pi}{3}$ [t] at 199 427
\pinlabel $a_1$ [r] at 164 449
\pinlabel $\verteq$ [r] at 160 437
\pinlabel $b_1$ [r] at 164 423
\pinlabel $b_4$ [b] at 134 498
\pinlabel $a_2$ [l] at 237 436
\pinlabel $b_3$ [r] at 63 436
\pinlabel $b_2$ [t] at 136 374
\pinlabel {\LARGE $\circlearrowleft$} at 199 436
\pinlabel {\LARGE $\circlearrowleft$} at 116 436
\vspace{12pt}
\endlabellist
 \subfloat[Case III]{\label{p2q4_8}\includegraphics[width=0.28\textwidth]{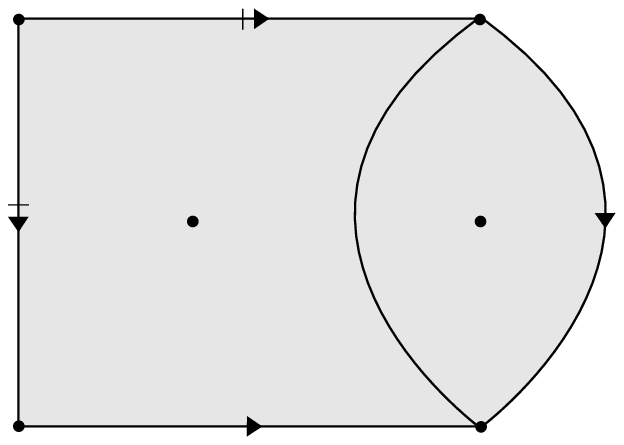}}
  \caption[\hspace{1em} Gluing patterns on $\partial X$ in Cases II and III]{Gluing patterns on $\partial X$ in Cases II and III}
 \end{figure}

$\bullet$ \textbf{Case III:} \emph{The edge $b_2$ is located as in the bottom right diagram of Figure \ref{p2q4_3}}. The identification pattern in this case is the one from Figure \ref{p2q4_8}. In particular, the graph $\Gamma$ in this case is the same as that of Case I (see Figure \ref{p2q4_6}). But we showed in Case I that this graph cannot occur, and so this case is impossible too.
\end{proof}

Using the ideas from the proof of Lemma \ref{lem_p2_q4}, it is easy to show that the case $p=2$ and $q=6$ is also \emph{impossible}.

\subsection{The Remaining Cases: $p=3$ and $p=4$} \label{subsect-remain-cases}

Arguing as we did for the case $p=2$, the reader can easily prove the two lemmas below. 

\begin{lem} \label{lem_p3_q3}
 If $p=3$, then $q=3$. Also, if $p=3$, then the disks $A$ and $B$ are  isometric, and each of them satisfies the following:
\begin{enumerate}
 \item its interior contains the vertex of the cone $T_aM$;
 \item its boundary consists of three geodesics in $T_aM$;
 \item one of its three corner angles equals $2\pi/3$.
\end{enumerate}
 Moreover, the hex sphere $M$ can be recovered from the (planar) isometric polygons from Figure \ref{summ_4}  by identifying pairs of edges on their boundaries as shown in Figure \ref{summ_4}. This figure also shows the only possible Voronoi graph $\Gamma$ when $p=q=3$.
\end{lem}

\begin{figure}[ht!]
\labellist
\small\hair 3pt
\pinlabel $\frac{4\pi}{3}$ [t] at 86 628
\pinlabel $\frac{\pi}{3}$ [tl] at -53 709
\pinlabel $\frac{\pi}{3}$ [tr] at 225 709
\pinlabel $P_A$ at 12 588
\pinlabel $\frac{4\pi}{3}$ [t] at 520 628
\pinlabel $\frac{\pi}{3}$ [tl] at 381 709
\pinlabel $\frac{\pi}{3}$ [tr] at 659 709
\pinlabel $P_B$ at 446 588
\pinlabel $\Gamma$ [b] at 304 402
\endlabellist
\centering
  \includegraphics[width=0.55\textwidth]{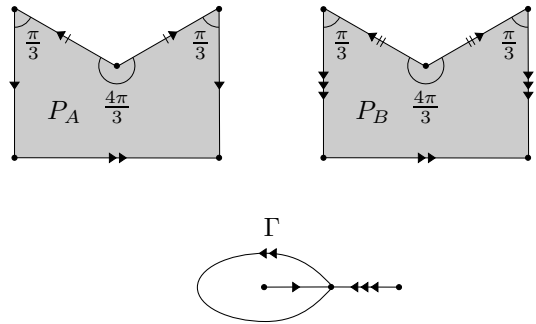}
\caption{Recovering $M$ from $P_A$ and $P_B$ when $p=q=3$}
  \label{summ_4}
\end{figure}

\begin{lem} \label{lem_p4_q4}
 If $p=4$, then $q=4$. Also, if $p=4$, then the disks $A$ and $B$ are  isometric, and each of them satisfies the following:
\begin{enumerate}
 \item its interior contains the vertex of the cone $T_aM$;
 \item its boundary consists of four geodesics in $T_aM$;
 \item one of its four corner angles equals $2\pi/3$.
\end{enumerate}
 Moreover, the hex sphere $M$ can be recovered from the (planar) isometric polygons shown in Figure \ref{summ_5}  by identifying pairs of edges on their boundaries as shown in Figure \ref{summ_5}. This figure also shows the only possible Voronoi graph $\Gamma$ when $p=q=4$.
\end{lem}

\begin{figure}[ht!]
\labellist
\small\hair 3pt
\pinlabel $\frac{4\pi}{3}$ [t] at 70 660
\pinlabel $\frac{\pi}{3}$ [tl] at -70 740
\pinlabel $\frac{\pi}{3}$ [tr] at 208 740
\pinlabel $P_A$ at 98 560
\pinlabel $\frac{4\pi}{3}$ [t] at 516 660
\pinlabel $\frac{\pi}{3}$ [tl] at 376 740
\pinlabel $\frac{\pi}{3}$ [tr] at 654 740
\pinlabel $P_B$ at 544 560
\pinlabel $\Gamma$ [b] at 270 359
\endlabellist
\centering
  \includegraphics[width=0.55\textwidth]{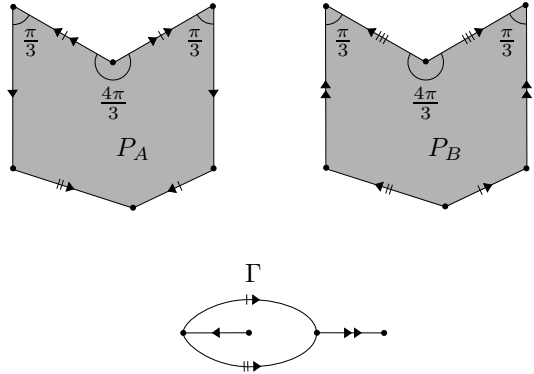}
\caption{Recovering $M$ from $P_A$ and $P_B$ when $p=q=4$}
  \label{summ_5}
\end{figure}
\section{Proving the Main Theorems} \label{sect-main-proofs}
We now prove the main theorems of this paper. 

\begin{thm} \label{thm-classif-Vor-graph}
Let $M$ be a hex sphere, and let $\Gamma$ be the Voronoi graph of $M$. Then, up to graph isomorphism, $\Gamma$ is one of the graphs from Figure \ref{prov-fig-classif-Vor-graph1}.

 \end{thm}

\begin{figure}[ht!]
\centering
  \includegraphics[width=0.45\textwidth]{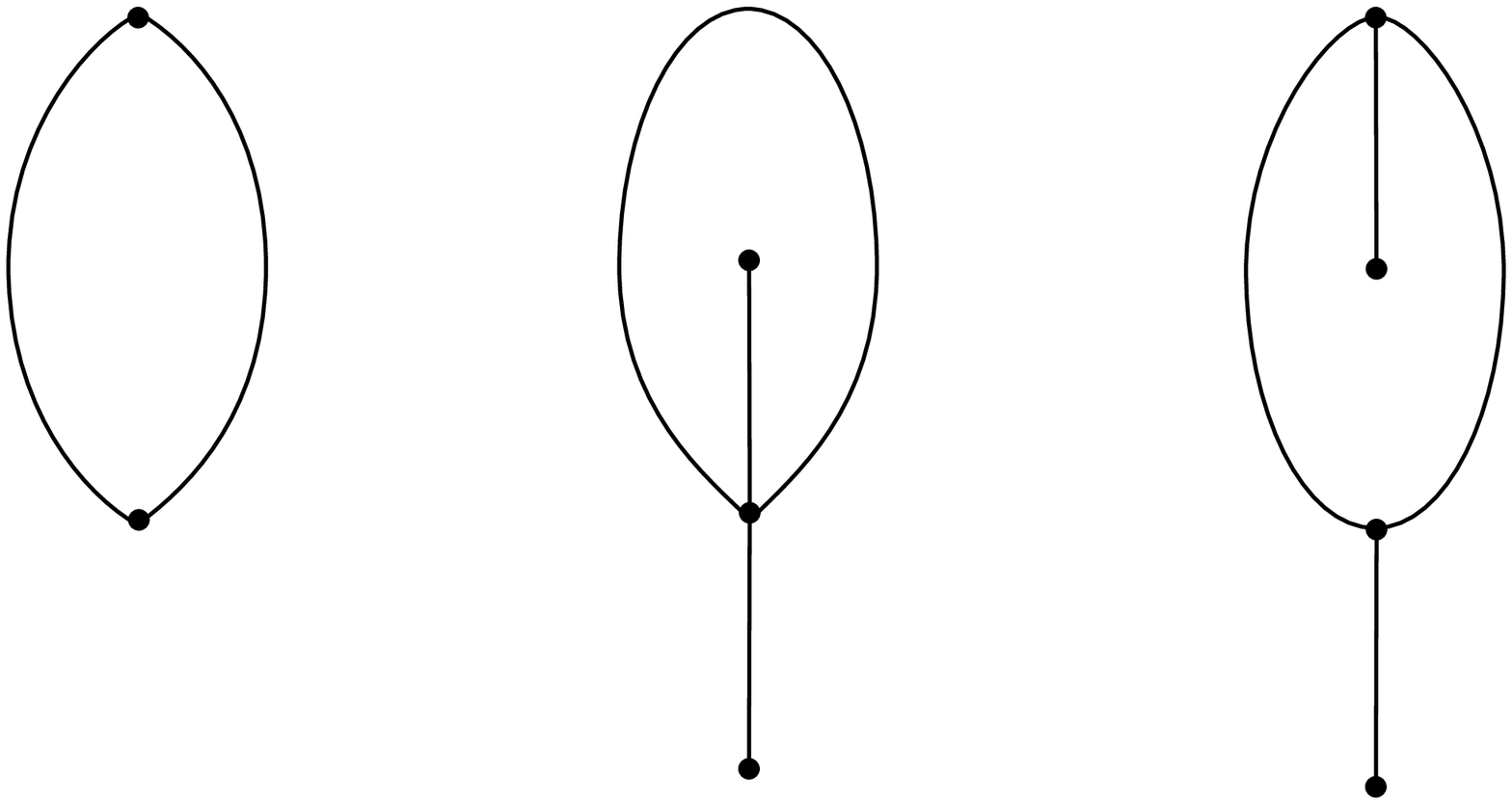}
  \caption{Classifying the Voronoi graphs of hex spheres}
  \label{prov-fig-classif-Vor-graph1}
\end{figure}

\begin{proof}
Let $p$ be the number of edges of the Voronoi region $A$. By Observation \ref{obs-pos-val-p}, the only possible values for $p$ are $2$, $3$ and $4$. If $p=2$, then $q=2$, and so Lemma \ref{lem_p2_q2} implies that the $\Gamma$ is the graph on the left of Figure \ref{prov-fig-classif-Vor-graph1}. If $p=3$, then Lemma \ref{lem_p3_q3} implies that $\Gamma$ is the graph on the middle of Figure \ref{prov-fig-classif-Vor-graph1}. Finally,  if $p=4$, then Lemma \ref{lem_p4_q4} implies that $\Gamma$ is the graph on the right of Figure \ref{prov-fig-classif-Vor-graph1}. 
\end{proof}

\begin{thm} \label{thm-describ-Vor-reg}
Let $M$ be a hex sphere, and let $A$ and $B$ its two Voronoi regions. Then 
\begin{enumerate}
 \item $A$ and $B$ are isometric.
 \item Each of $A$ and $B$ embeds isometrically in a Euclidean cone as a convex geodesic polygon,  with the center of the Voronoi region corresponding to the vertex of the cone.
 \item The hex sphere $M$ can be recovered from the disjoint union of $A$ and $B$ by identifying pairs of edges on their boundaries according to one of $3$ possible combinatorial patterns.
 \end{enumerate}
\end{thm}

\begin{proof}
$(1)$ and $(2)$ follow from Lemmas \ref{lem_p2_q2}, \ref{lem_p3_q3}, \ref{lem_p4_q4} and Proposition \ref{Vor_emb_in_con}, respectively. If $p=2$, then Lemma \ref{lem_p2_q2} implies that $M$ can be recovered from the planar polygons from Figure \ref{summ_2}  by identifying pairs of edges on their boundaries as shown in Figure \ref{summ_2}. Each Voronoi region is obtained from one of these planar polygons by identifying the two sides that are incident to the only vertex of angle $4\pi/3$. Therefore, $M$ can be recovered from its Voronoi regions by identifying pairs of edges on their boundaries. This same conclusion is also true for $p=3$ and $p=4$ (by Lemmas \ref{lem_p3_q3} and \ref{lem_p4_q4}).
\end{proof}

\section{Concluding Remarks}

A hex sphere is a singular Euclidean sphere with $4$ cones whose cone angles are (integer) multiples of $\frac{2\pi}{3}$ but less than $2\pi$. Given a hex sphere $M$, we considered its Voronoi decomposition centered at the two cone points with greatest cone angles. In this paper we used elementary Euclidean geometry to describe geometrically the Voronoi regions of hex spheres. In particular, we showed that the two Voronoi regions of a hex sphere are always isometric. We also classified the Voronoi graphs of hex spheres. Finally, we gave all possible ways to reconstruct hex spheres from suitable polygons in the Euclidean plane. However, to prove all these things, we did a long and inelegant case-by-case analysis of all possible numbers of edges on the boundaries of the Voronoi cells. This makes one wonder about the existence of more direct and elegant proofs of these results. Perhaps one way to shorten the proofs of these results is using Riemannian metrics to approximate hex metrics (this was suggested by Daryl Cooper).

\bibliographystyle{alpha}
\bibliography{biblio}

\end{document}